\documentclass[letterpaper, 10 pt, conference]{ieeeconf} 

\usepackage{cite}
\usepackage{amsmath,amssymb,amsfonts}
\usepackage{algorithmic}
\usepackage{graphicx}
\usepackage{textcomp}

\usepackage{subfigure}
\usepackage{epsfig} 
\usepackage{hyperref}
\hypersetup{colorlinks=false, hidelinks}
\usepackage{xcolor}
\newenvironment{IEEEkeywords}{\vspace{0.5em}\noindent{\bfseries\itshape Index Terms---}}{\par}
\let\IEEEPARstart\PARstart

\newtheorem{theorem}{Theorem}
\newtheorem{corollary}[theorem]{Corollary}

\newtheorem{lemma}[theorem]{Lemma}
\newtheorem{proposition}[theorem]{Proposition}
\newtheorem{remark}{Remark}
\newtheorem{problem}{Problem}
\newtheorem{assumption}{Assumption}

\newcommand{\real}{\mathbb{R}}

\newcommand{\Cc}{{\mathcal{C}}}

\newcommand{\Rc}{{\mathcal{R}}}

\newcommand{\Kc}{{\mathcal{K}}}

\newcommand{\bx}{{\mathbf{x}}}

\newcommand{\by}{{\mathbf{y}}}

\newcommand{\bz}{{\mathbf{z}}}

\newcommand{\bg}{{\mathbf{g}}}

\newcommand{\bu}{{\mathbf{u}}}
\newcommand{\bp}{{\mathbf{p}}}

\newcommand{\bv}{{\mathbf{v}}}
\newcommand{\bc}{{\mathbf{c}}}
\newcommand{\bb}{{\mathbf{b}}}

\newcommand{\bk}{{\mathbf{k}}}
\newcommand{\bbf}{{\mathbf{f}}}

\newcommand{\bA}{{\mathbf{A}}}
\newcommand{\bB}{{\mathbf{B}}}

\newcommand{\bG}{{\mathbf{G}}}
\newcommand{\bK}{{\mathbf{K}}}

\newcommand{\bI}{{\mathbf{I}}}

\newcommand{\bP}{{\mathbf{P}}}

\newcommand{\bC}{{\mathbf{C}}}

\newcommand{\bQ}{{\mathbf{Q}}}
\newcommand{\bR}{{\mathbf{R}}}
\newcommand{\bY}{{\mathbf{Y}}}

\newcommand{\bDelta}{{\boldsymbol{\Delta}}}

\newcommand\xqed[1]{%
  \leavevmode\unskip\penalty9999 \hbox{}\nobreak\hfill
  \quad\hbox{#1}}
\newcommand\demo{\xqed{$\bullet$}}

\newcommand{\longthmtitle}[1]{\mbox{}\emph{(#1):}}
\newcommand{\setdef}[2]{\{#1 : #2\}}

\newcommand{\norm}[1]{\left\lVert#1\right\rVert}

\allowdisplaybreaks
\renewcommand{\baselinestretch}{.98}

\begin{document}
\title{\LARGE \bf Dynamical Properties of Safety Filters for Linear Systems \\
and Affine Control Barrier Functions}
\author{Pol Mestres, Shima Sadat Mousavi, and Aaron D. Ames
\thanks{The authors are with the Department of Mechanical and Civil
Engineering, California Institute of Technology, Pasadena, CA 91125, USA.
Emails: \texttt{mestres,smousavi,ames@caltech.edu}.
This research is supported by The Boeing Company.}
}
\maketitle
\thispagestyle{empty}

\begin{abstract}
This letter studies the dynamical properties of \textit{safety filters} designed based on Control Barrier Functions (CBF). This mechanism, which is popular in safety-critical 
applications, takes a nominal controller and minimally modifies it to render it safe.
Although CBF-based safety filters make the closed-loop system safe, characterizing their additional dynamical properties, such as stability, boundedness, or existence of spurious equilibria, remains a challenging problem.
Here, we address this problem for the case of linear systems and an affine CBF constraint. We provide conditions under which the closed-loop system presents undesired equilibria, unbounded trajectories, or the origin is globally exponentially stable.
\end{abstract}


\section{Introduction}

\IEEEPARstart{M}{odern} engineering systems such as aerospace vehicles and humanoid robots
are subject to strict safety requirements during operation.
Control theory formalizes such safety specifications by requiring that the state of the system remains within a prescribed region.
Different techniques have been developed to achieve this, including control barrier functions (CBFs)~\cite{ADA-XX-JWG-PT:17}, model predictive control~\cite{JBR-DQM-MMD:17}, Hamilton--Jacobi reachability~\cite{SB-MC-SH-CJT:17}, and reference governors~\cite{EG-SDC-IK:17}.

Here we focus on a particular approach to design safe controllers referred to as \textit{CBF-based safety filters}. This mechanism minimally modifies a nominal 
(potentially unsafe) 
controller at every 
state in order to make it safe.
The resulting safe controller can be obtained at every point as the solution of an optimization problem, which is a quadratic program (QP) if the dynamics are control-affine~\cite{ADA-SC-ME-GN-KS-PT:19}, in which case it can be expressed in closed-form~\cite{XT-DVD:24,PM-SSM-PO-LY-ED-JWB-ADA:25, mousavi2025vertices,mousavi2026structure}.

Despite the widespread application of CBF-based safety filters for a wide range of tasks such as adaptive cruise control~\cite{ADA-JWG-PT:14}, bipedal robotic walking~\cite{SCH-XX-ADA:15}, or aircraft flight~\cite{OS-ZS-MM-JG-KR-NR-CF:24}, the recent works~\cite{PM-YC-EDA-JC:25-arxiv,NM-JC-PS-KZ:25,XT-DVD:24,MFR-APA-PT:21} 
have shown that the closed-loop system obtained from CBF-based safety filters can exhibit a variety of undesired behaviors, such as unbounded trajectories, limit cycles, or undesired equilibria (i.e., equilibria that do not exist in the nominal system).
\textcolor{black}{In particular,~\cite{XT-DVD:24} studies the problem by adding a control Lyapunov function (CLF)
constraint on the safety filter, which introduces a different set of undesired equilibria, and only local asymptotic stability of desired equilibria is shown. In this paper, we study the safety filter without the CLF constraint.}
In general, 
a full characterization of the dynamical properties produced by CBF-based safety filters, as well as design principles that ensure that they lead to desirable dynamical behaviors remains an important open research problem.

In this paper, we study the problem in the case where the dynamics are linear and the safety constraints are affine in the
system state.
Although there exist computationally tractable techniques for reachability analysis of linear systems~\cite{MA:20,ABK-PV:00}, solving constrained optimal control problems is often still difficult~\cite{SB-MC-SH-CJT:17}, and hence justifies the use of CBF-based safety filters.
Additionally, this setting is particularly relevant in applications such as aerospace control~\cite{EL-KAW:24}, where control designs are implemented to linearized models of the aircraft's full dynamics.
Moreover, the development of these results in the linear case paves the way to establishing analogous results for general nonlinear systems.

\begin{figure}[t]
    \centering
    \includegraphics[width=0.9\columnwidth]{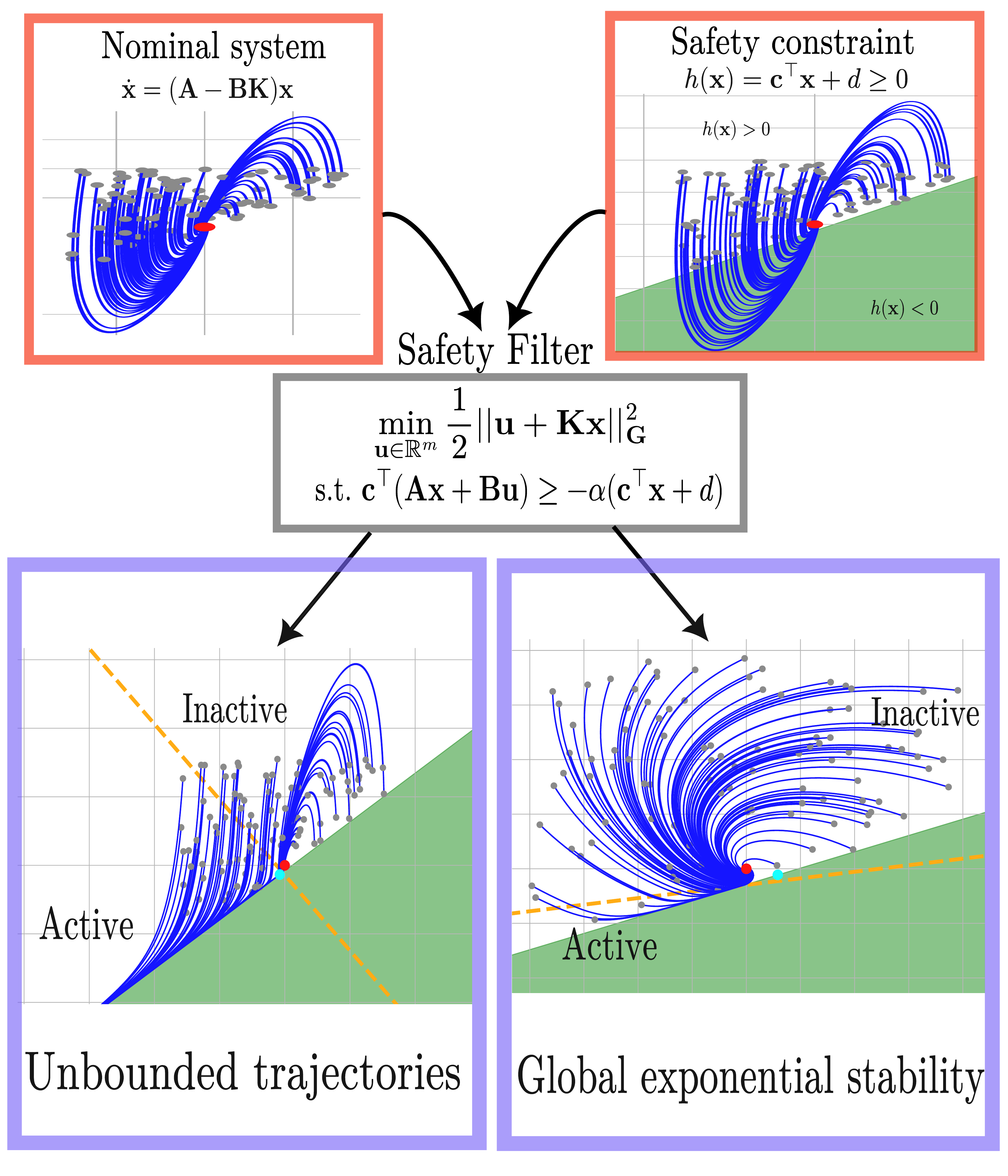}
    \caption{
    \footnotesize
    Overview of the paper. 
    A CBF-based safety filter modifies stable nominal linear dynamics under an affine constraint.
    Depending on the active-mode spectrum and its compatibility with the nominal
mode, the resulting piecewise-affine closed loop may be globally exponentially
stable, exhibit undesired equilibria, or have unbounded trajectories.
    }
    \label{fig:hero}
\end{figure}


\vspace{-0.01mm}

The contributions of the paper are as follows.
First, by exploiting the closed-form solution of the safety filter, 
we show how the different parameters
affect the eigenstructure of the matrix defining the dynamics when the filter is active (i.e., the active-mode matrix), which plays an important role in the results of the paper.
Second, we provide a characterization of the set of undesired equilibria, and show that their existence is related to the parity of the number of positive real eigenvalues of this matrix.
Third, we provide spectral conditions, 
involving the active-mode matrix and its compatibility with the nominal mode, that ensure that the origin is globally exponentially stable (GES). We also show that if this matrix has positive real eigenvalues, then the closed-loop system has unbounded trajectories. 
Fourth, we interpret these conditions through the invariant zeros of an appropriately defined linear single-input single-output (SISO) system. 
Finally, we show that designing a nominal linear controller that renders the filtered system GES can be cast as a linear matrix inequality (LMI).
We illustrate our results in different simulations.

\section{Background}\label{sec:background}

Here we revisit basic notation, CBFs, and introduce the problem we seek to solve in this paper.

\subsubsection*{Notation}
We denote by $\mathbb{N}$, $\real$, $\real_{\geq0}$ the set of natural, real and non-negative real numbers, respectively. We use bold (resp. non-bold) symbols to represent vectors (resp. scalars).
Given $n\in\mathbb{N}$, we let $\mathbf{0}_n$ be the $n$-dimensional zero vector and $[n] = \{ 1, \hdots, n \}$.
For a set $\mathcal{S}\subset\real^n$, $\text{Int}(\mathcal{S})$ and $\partial\mathcal{S}$ denote the interior and boundary of the set $\mathcal{S}$.
For a matrix $\bG\in\real^{n\times n}$, $\text{det}(\bG)$, $\text{adj}(\bG)$, and $\text{spec}(\bG)$ denote its determinant, adjoint matrix, and spectrum, respectively.
For $\bx\in\real^n$, 
and a positive-definite matrix $\bG \in \real^{n\times n}$, $\norm{\bx}_{\bG} = \sqrt{\bx^\top \bG \bx}$.
A function {\color{black}$\alpha:\real_{\geq0}\to\real_{\geq0}$} is of class $\Kc$ if it is continuous, strictly increasing, and satisfies $\alpha(0) = 0$.
Given a linear system $\dot{\bx} = \bA\bx+\bB\bu$, with $\bx\in\real^n$, $\bu\in\real^m$, and output $\by = \bC \bx$, with $\by\in\real^p$, we say that $\bz\in\mathbb{C}$ is an invariant zero of this linear system $(\bA,\bB,\bC)$ if $\text{rank}\begin{pmatrix}
    \bA-\bz\bI & \bB \\
    \bC & \mathbf{0}
\end{pmatrix} < n + m$.
For $z\in\real$,
$\text{sign}(z) = 1$ if $z > 0$, $\text{sign}(z) = -1$ if $z < 0$, and $\text{sign}(z) = 0$ if $z = 0$.
We say that the matrices $\bA_1, \bA_2$ have a common quadratic Lyapunov function (CQLF) if
$\dot{\bx} = \bA_1\bx$, $\dot{\bx} = \bA_2\bx$ have one such function.

Consider a control-affine system: 
\begin{align}\label{eq:control-affine-system}
    \dot{\bx} = \bbf(\bx) + \bg(\bx) \bu,
\end{align}
where $\bx\in\real^n$ is the state, $\bu\in\real^m$ the control input, and $\bbf:\real^n\to\real^n$, $\bg:\real^n\to\real^{n\times m}$ are locally Lipschitz.
Let $h:\real^n\to\real$ be continuously differentiable, and $\Cc = \setdef{\bx\in\real^n}{h(\bx) \geq 0}$ be a safe set.
We are interested in designing a controller that renders $\Cc$ forward invariant.
If $h$ has relative degree one, this can be achieved through control barrier functions (CBFs)~\cite{ADA-SC-ME-GN-KS-PT:19}. 
In the case where $h$ has arbitrary relative degree, CBFs are generalized through high-order control barrier functions (HOCBFs)~\cite{WX-CB:22}.
The following result recalls its construction and the safety guarantee they provide.

\begin{theorem}\longthmtitle{HOCBF~\cite[{\color{black}Theorem} 4]{WX-CB:22}}\label{thm:hocbf}
    Let $h:\real^n\to\real$ be a continuously differentiable function defining a set $\Cc = \Cc_0 = \setdef{\bx\in\real^n}{h(\bx) \geq 0}$.
    Suppose that $h$ has relative degree $r\in\mathbb{N}$ in $\real^n$
    and define $h_0(\bx) = h(\bx)$ and $h_i(\bx) = \dot{h}_{i-1}(\bx) + {\color{black}\bar{\alpha}_i}(h_{i-1}(\bx))$ for $i\in[r-1]$ with {\color{black}$\bar{\alpha}_i$} an extended class $\Kc$ function.
    Further let $\Cc_i = \setdef{\bx\in\real^n}{h_i(\bx) \geq 0}$ for $i\in[r-1]$.
    Then,
    \begin{align}\label{eq:hocbf-inequality}
        L_{\bbf}h_{r-1}(\bx) + L_{\bg}h_{r-1}(\bx)\bu + {\color{black}\bar{\alpha}_r}(h_{r-1}(\bx)) \geq 0,
    \end{align}
    (with $\alpha_r$ a class $\Kc$ function) is feasible for all $\bx\in\real^n$ and any locally Lipschitz controller satisfying~\eqref{eq:hocbf-inequality} in $\bar{\Cc} := \cap_{i=0}^{r-1} \Cc_i$ renders $\bar{\Cc}$ forward invariant.
\end{theorem}

{\color{black} We note that the fact that~\eqref{eq:hocbf-inequality} is feasible for all $\bx\in\real^n$ is guaranteed by the fact that $h$ has relative degree $r$ in $\real^n$ (and hence, $L_{\bg} L_{\bbf}^{r-1} h(\bx) \neq \mathbf{0}_m$ for all $\bx\in\real^n$).}
As shown in {\color{black}Theorem}~\ref{thm:hocbf}, the key property of HOCBFs is that any locally Lipschitz controller satisfying~\eqref{eq:hocbf-inequality} renders the set $\bar{\Cc} \subset \Cc$ forward invariant.
A common mechanism to design such controller is through \textit{safety filters}. Given a nominal controller $\bk:\real^n\to\real^m$ (often with desirable properties such as asymptotic stability or optimality) and $\bG \succ 0$, safety filters construct a safe controller $\bu^*$ as follows:
\begin{align}\label{eq:ustar}
    \bu^*(\bx) = \text{arg}\min\limits_{\bu\in\real^m} \frac{1}{2}\norm{\bu - \bk(\bx)}_{\bG}^2, \
    &\text{s.t.}~\eqref{eq:hocbf-inequality}.
\end{align}
By construction, $\bu^*$ satisfies the HOCBF condition~\eqref{eq:hocbf-inequality}. Furthermore, as shown in~\cite{XX-PT-JWG-ADA:15,PM-AA-JC:25-ejc} under mild assumptions $\bu^*$ is locally Lipschitz, and hence renders $\bar{\Cc}$ forward invariant.
As noted in~\cite{PM-YC-EDA-JC:25-arxiv,MFR-APA-PT:21,XT-DVD:24}, characterizing dynamical properties such as boundedness of trajectories, regions of attraction, or existence of undesired equilibria for the system obtained by using $\bu^*$ in~\eqref{eq:control-affine-system} is challenging.
In this paper, we consider the case where~\eqref{eq:control-affine-system} is linear and $h$ is affine:

\begin{assumption}\label{as:linear-dynamics-affine-constraints}
    The dynamics~\eqref{eq:control-affine-system} are linear{\color{black}:} $\bbf(\bx) = \bA \bx$, $\bg(\bx) = \bB$, with $\bA \in \real^{n\times n}$, $\bB\in\real^{n\times m}$, and $h$ is affine, i.e., $h(\bx) = \bc^\top \bx + d$, with $\bc\in\real^n$, $d\in\real$.
\end{assumption}

We let {\color{black}
$r\in\mathbb{N}$} be the relative degree of $h$
under Assumption~\ref{as:linear-dynamics-affine-constraints}
and make the following two additional assumptions.

\begin{assumption}\label{as:origin-interior}
    $\mathbf{0}_n\in\text{Int}(\Cc)$. Equivalently, $d > 0$.
\end{assumption}


\begin{assumption}\label{as:stabilizability}
    The pair $(\bA,\bB)$ is stabilizable.
\end{assumption}

In the rest of the paper, we assume Assumptions~\ref{as:linear-dynamics-affine-constraints}-\ref{as:stabilizability}.
Since $\bk$ is often taken to be a controller that stabilizes the origin, Assumption~\ref{as:origin-interior} requires that the point where we wish to stabilize the system is safe, which is common in practice.
Although the results in the paper can be generalized to the case where the origin is not in the safe set we make Assumption~\ref{as:origin-interior} to simplify our analysis.

In what follows, we take the class $\Kc$ functions {\color{black}$\{ \bar{\alpha}_i \}_{i=1}^r$} to be linear with slopes $\{ \alpha_i > 0 \}_{i=1}^r$.
Further, we let $\alpha := \prod_{i=1}^r \alpha_i$. 
The following result provides an explicit expression for the functions $\{ h_i \}_{i=1}^{r-1}$ and inequality~\eqref{eq:hocbf-inequality} 
defined in {\color{black}Theorem}~\ref{thm:hocbf}
under Assumption~\ref{as:linear-dynamics-affine-constraints}.
\begin{lemma}\longthmtitle{HOCBF for linear dynamics and constraints}\label{lem:hocbf-linear-dynamics-constraints}
    Let $\phi_i(s) = \prod_{j=1}^i (s + \alpha_j)$ for $i\in[r]$ and $\phi = \phi_r$.
    Then, $h_i(\bx) = \bc^\top \phi_i(\bA) \bx + d \prod_{j=1}^i \alpha_j$ for $i\in[r-1]$.
    Furthermore,~\eqref{eq:hocbf-inequality} reads
    \begin{align}\label{eq:hocbf-inequality-linear}
        \bc^\top \phi(\bA)\bx + \bc^\top \bA^{r-1} \bB \bu + \alpha d \geq 0.
    \end{align}
\end{lemma}
\begin{proof}
    It is simple to check that the expression for $h_1(\bx)$ holds. Now, suppose $h_i(\bx) = \bc^\top \phi_i(\bA)\bx + d \prod_{j=1}^i \alpha_j$ for $i\in[r-2]$. Then, 
    \begin{align*}
        &h_{i+1}(\bx) \! = \! \bc^\top \phi_i(\bA) (\bA\bx \! + \! \bB\bu) \! + \! \alpha_{i+1}(\bc^\top \phi_i(\bA)\bx \! + \! d \prod_{j=1}^i \alpha_j) \\
        &\! = \! \bc^\top \! \phi_i(\bA) (\bA \! + \! \alpha_{i+1}\bI) \bx \! + \! d\prod_{j=1}^{i+1} \alpha_j \! = \! \bc^\top \! \phi_{i+1}(\bA)\bx \! + \! d\prod_{j=1}^{i+1} \! \alpha_j,
    \end{align*}
    where we have used the fact that since $\bc^\top \bx + d$ has relative degree $r$, $\bc^\top \bA^i \bB = 0$ for $i\in[r-2]$, and $\phi_i(\bA)$ has powers of $\bA$ only up to $\bA^i$.
    By induction, this proves that $h_i(\bx) = \bc^\top \phi_i(\bA) \bx + d \prod_{j=1}^i \alpha_j$ for $i\in[r-1]$.
    Now, the fact that~\eqref{eq:hocbf-inequality} reduces to~\eqref{eq:hocbf-inequality-linear} follows from the expression of $h_{r-1}$. 
\end{proof}

We further consider~\eqref{eq:ustar} equipped with a linear nominal controller $\bk(\bx) = -\bK \bx$ that makes $\bA_0 := \bA-\bB\bK$ Hurwitz.
Since $\bc^\top \bx + d$ has relative degree $r$, $\bc^\top \bA^{r-1}\bB \neq \mathbf{0}_m$ and~\eqref{eq:hocbf-inequality-linear} is feasible for all $\bx\in\real^n$. By an argument analogous to that of~\cite[{\color{black}Theorem} 2]{XX-PT-JWG-ADA:15}, $\bu^*$ is locally Lipschitz. This implies that closed-loop solutions are unique and $\bar{\Cc}$ is forward invariant.
%
Formally, this is the problem we seek to solve in this paper:

\begin{problem}\label{prob:problem}
    Given system~\eqref{eq:control-affine-system} and controller $\bu^*$ under assumptions~\ref{as:linear-dynamics-affine-constraints}-\ref{as:stabilizability},
    characterize the dynamical properties of the corresponding closed-loop system, including 
    the set of equilibria,
    the stability properties of the origin and boundedness of trajectories.
\end{problem}

In order to solve Problem~\ref{prob:problem}, it is useful to write the closed-loop system explicitly.
To do so, let us define the auxiliary constants 
$\theta = \norm{\bB^\top (\bA^\top)^{r-1}\bc}_{\bG^{-1}}^2$,
\begin{align*}
    \bv_1 = -\frac{\bB \bG^{-1} \bB^\top (\bA^\top)^{r-1} \bc }{ \theta }, \ \bv_2^\top = \bc^\top (\phi(\bA)-\bA^{r-1}\bB\bK),
\end{align*}

Now, by writing the controller $\bu^*$ in closed-form (cf.~\cite{PM-YC-EDA-JC:25-arxiv,PM-SSM-PO-LY-ED-JWB-ADA:25}), the closed-loop system can be obtained as

\begin{align}\label{eq:closed-loop-piecewise-affine}
    \dot{\bx} = \begin{cases}
        \bA_0 \bx \quad &\text{if} \ \eta(\bx) \geq 0, \\
        \tilde{\bA}\bx + \tilde{\bb}, \quad &\text{else},
    \end{cases}
\end{align}
where $\eta(\bx) = \bv_2^\top \bx + \alpha d$,
$\tilde{\bA} = \bA_0 + \bv_1 \bv_2^\top$, and 
$\tilde{\bb} = \alpha d \bv_1$.
In the sequel, we let 
$\Rc_{+} = \setdef{\bx\in\real^n}{\eta(\bx) \geq 0}$
and
$\Rc_{-} = \real^n\backslash\Rc_{+}$.
%
Note that~\eqref{eq:closed-loop-piecewise-affine} is a piecewise-affine system.
%

\section{Undesired Equilibria}\label{sec:undesired-equilibria}

In this section we study the set of equilibria of~\eqref{eq:closed-loop-piecewise-affine}.
First note that since $\eta(\mathbf{0}_n) = \alpha d$, by Assumption~\ref{as:origin-interior}, $\eta(\mathbf{0}_n) > 0$ and therefore the origin is an equilibrium of~\eqref{eq:closed-loop-piecewise-affine}.
However, as noted in~\cite{XT-DVD:24,MFR-APA-PT:21,PM-YC-EDA-JC:25-arxiv}, CBF-based safety filters can introduce undesired equilibria in the corresponding closed-loop system.
In order to characterize such equilibria, we first provide different facts about the eigenstructure of $\tilde{\bA}$.

\begin{lemma}\longthmtitle{Eigenstructure of $\tilde{\bA}$}\label{lem:conditions-tildeA-invertible}
    The value $-\alpha_r$ is an eigenvalue of $\tilde{\bA}$ with left eigenvector $\bc^\top \phi_{r-1}(\bA)$.
    Moreover, the eigenvalues of $\tilde{\bA}$ that are not eigenvalues of $\bA_0$ are either in $\{ -\alpha_i \}_{i=1}^{r-1}$ or satisfy
    \begin{align}\label{eq:eigenvalues-of-tildeA}
        \bc^\top (\lambda \bI_n - \bA_0)^{-1} \bB  \bG^{-1} \bB^\top (\bA^\top)^{r-1} \bc = 0.
    \end{align}
\end{lemma}
\begin{proof}
    Since $\tilde{\bA} = \bA_0 + \bv_1 \bv_2^\top$, $\tilde{\bA}$ is a rank-one update of $\bA_0$. By the matrix determinant lemma (cf.~\cite[Fact 2.16.3]{DSB:09}), 
    if $\lambda$ is an eigenvalue of $\tilde{\bA}$, then it is either an eigenvalue of $\bA_0$ or 
    satisfies $1 - \bv_2^\top (\lambda \bI - \bA_0)^{-1} \bv_1 = 0$.
    After some manipulations, this condition is equivalent to
    \begin{align}\label{eq:matrix-determinant-lemma-2}
        \bc^\top (\phi(\bA)-\bA^r + \bA^{r-1}\lambda)(\lambda \bI - \bA_0)^{-1} \bv_1 = 0.
    \end{align}
    Now, let us show that for $i\in[r-1]$,
    \begin{align}\label{eq:reduce-degree-identity}
        \bc^\top \bA^i (\lambda \bI - \bA_0)^{-1} \bB = \lambda \bc^\top \bA^{i-1} (\lambda \bI - \bA_0)^{-1} \bB.
    \end{align}
    Indeed, this follows from the identity
    \begin{align*}
        \bA_0(\lambda\bI-\bA_0)^{-1} = \lambda (\lambda\bI - \bA_0)^{-1} -\bI,
    \end{align*}
    after left multiplying it by $\bc^\top \bA^{i-1}$, right multiplying it by $\bB$ and using the relative degree assumption. 
    Now, 
    if we let $\phi(s) = s^r + \sum_{i=1}^{r-1} \beta_i s^i$, 
    we get that~\eqref{eq:matrix-determinant-lemma-2} is equivalent to 
    \begin{align*}
        \lambda \bc^\top \bA^{r-1}(\lambda\bI - \bA_0)^{-1} \bv_1 \! + \! \sum_{i=1}^{r-1} \beta_i \bc^\top \bA^i (\lambda\bI-\bA_0)^{-1}\bv_1 = 0.
    \end{align*}
    After repeatedly applying~\eqref{eq:reduce-degree-identity} to each term we get that that~\eqref{eq:matrix-determinant-lemma-2} is equivalent to
    $\phi(\lambda) \bc^\top (\lambda \bI_n - \bA_0)^{-1} \bB  \bG^{-1} \bB^\top (\bA^\top)^{r-1} \bc = 0$,
    from where it follows that the eigenvalues of $\tilde{\bA}$ that are not $\{ -\alpha_i \}_{i=1}^r$ and and are not eigenvalues of $\bA_0$, satisfy~\eqref{eq:eigenvalues-of-tildeA}.
    Finally, the fact that $\bc^\top \phi_{r-1}(\bA)$ is the left eigenvector associated with $-\alpha_r$ follows from a direct computation. 
    %
    %
\end{proof}

From~\eqref{eq:eigenvalues-of-tildeA}, Lemma~\ref{lem:conditions-tildeA-invertible} implies that $\tilde{\bA}$ is invertible 
if $\xi := \bc^\top \bA_0^{-1} \bB \bG^{-1} \bB^\top (\bA^\top)^{r-1} \bc \neq 0$.
The following result leverages Lemma~\ref{lem:conditions-tildeA-invertible} to characterize the set of equilibria of~\eqref{eq:closed-loop-piecewise-affine}.
\begin{proposition}\longthmtitle{Undesired equilibria}\label{prop:existence-of-undesired-equilibrium}
    If $\xi \neq 0$, then $\text{sign}( \eta(-\tilde{\bA}^{-1} \tilde{\bb}) ) = -\text{sign}( \xi )$.
    Furthermore, {\color{black}by letting $\mathcal{E}$ be the set of equilibria of~\eqref{eq:closed-loop-piecewise-affine},}
    \begin{enumerate}
        \item if $\eta(-\tilde{\bA}^{-1} \tilde{\bb}) > 0$ 
        (or equivalently, $\xi < 0$), 
        {\color{black}$\mathcal{E} = \{ \mathbf{0}_n \}$;}
        \item if $\eta(-\tilde{\bA}^{-1} \tilde{\bb}) < 0$ (or equivalently, $\xi > 0$), {\color{black}$\mathcal{E} = \{ \mathbf{0}_n \} \cup \{ -\tilde{\bA}^{-1} \tilde{\bb} \}$;}
        \item {\color{black}if $\xi = 0$, 
        $\mathcal{E} = \{ \mathbf{0}_n \} \cup \setdef{\bx\in\real^n}{\tilde{\bA}\bx = -\tilde{\bb}, \ \eta(\bx) < 0}$. In particular, if $\tilde{\bb} \in \text{Im}(\tilde{\bA})$, $\mathcal{E}$ contains infinite points, whereas if $\tilde{\bb} \notin \text{Im}(\tilde{\bA})$, $\mathcal{E} = \{ \mathbf{0}_n \}$.}
    \end{enumerate}
\end{proposition}
\begin{proof}
    Since $\bA_0$ is Hurwitz, the only possible equilibrium in $\Rc_{+}$ is the origin. First assume $\xi \neq 0$. By Lemma~\ref{lem:conditions-tildeA-invertible},~$\tilde{\bA}$ is invertible, and the only possible equilibrium in $\Rc_{-}$ is $\tilde{\bp} = -\tilde{\bA}^{-1} \tilde{\bb}$.
    However, in order for $\tilde{\bp}$ to be an 
    equilibrium of~\eqref{eq:closed-loop-piecewise-affine}, we need that $\tilde{\bp}\in\Rc_{-}$, i.e., $\eta(-\tilde{\bA}^{-1} \tilde{\bb}) < 0$, in which case {\color{black} $\mathcal{E} = \{ \mathbf{0}_n \}\cup \{ -\tilde{\bA}^{-1} \tilde{\bb} \}$.}
    Alternatively, if $\eta(-\tilde{\bA}^{-1} \tilde{\bb}) \geq 0$, 
    {\color{black}$\mathcal{E} = \{ \mathbf{0}_n \}$.}
    On the other hand, if $\xi = 0$, by Lemma~\ref{lem:conditions-tildeA-invertible} the matrix $\tilde{\bA}$ is not invertible and therefore 
    {\color{black}
    $\mathcal{E} = \{ \mathbf{0}_n \} \cup \setdef{ \bx\in\real^n }{ \tilde{\bA} \bx = -\tilde{\bb}, \ \eta(\bx) < 0}$.
    To see that $\mathcal{E}$ contains infinite points if $\tilde{\bb}\in\text{Im}(\tilde{\bA})$, we note that $\setdef{\bx\in\real^n}{\tilde{\bA}\bx = -\tilde{\bb}}$ is a subspace of dimension equal to the dimension of $\ker(\tilde{\bA})$), and the only way it does not intersect with $\setdef{\bx\in\real^n}{\eta(\bx) < 0}$ is if the vectors in $\ker(\tilde{\bA})$ are orthogonal to $\bv_2$ 
    (the normal vector to the hyperplane $\eta(\bx) = 0$). However, if $\bv\in\real^n$ is such that $\tilde{\bA}\bv = \mathbf{0}_n$ and $\bv_2^\top \bv = 0$, we have that $\bA_0 \bv = \mathbf{0}_n$, which is a contradiction because $\bA_0$ is Hurwitz. Hence, if $\tilde{\bb} \in \text{Im}(\tilde{\bA})$, $\mathcal{E}$ contains infinite points, whereas if $\tilde{\bb} \notin \text{Im}(\tilde{\bA})$, $\mathcal{E} = \{ \mathbf{0}_n \}$.}
    Now we show that if $\xi \neq 0$, then
    $\text{sign}( \eta(-\tilde{\bA}^{-1} \tilde{\bb}) ) = -\text{sign}( \xi )$.
    By using
    $\bA \bA_0^{-1} = \bI + \bB \bK \bA_0^{-1}$ and the relative
    degree condition, we have 
    \begin{align}\label{eq:simplification-expression}
        \bv_2^\top \bA_0^{-1} = \bc^\top \bA^{r-1} + \alpha \bc^\top \bA_0^{-1} + \sum_{j=1}^{r-1} \beta_j \bc^\top \bA^{j-1},
    \end{align}
    where $\{ \beta_i \}_{i=1}^{r-1}$ are as defined in the proof of Lemma~\ref{lem:conditions-tildeA-invertible}.
    Now, as shown in the proof of Lemma~\ref{lem:conditions-tildeA-invertible}, $\tilde{\bA}$ is a rank-one update of $\bA_0$,
    and therefore $\tilde{\bA}^{-1}$ can be computed using  the Sherman-Morrison formula (cf.~\cite[Fact 2.16.3]{DSB:09}):
    \begin{align}\label{eq:sherman-morrison}
        \tilde{\bA}^{-1} = \bA_0^{-1}\Big(
        \bI - \frac{\bv_1 \bv_2^\top \bA_0^{-1} }{ 1 + \bv_2^\top \bA_0^{-1} \bv_1 }
        \Big).
    \end{align}
    By using~\eqref{eq:simplification-expression} in~\eqref{eq:sherman-morrison}, we get
    $\tilde{\bA}^{-1} = \bA_0^{-1}(\bI - \bDelta)$,
    where 
    \begin{align*}
        \bDelta \! = \! \frac{ \bB \bG^{-1}\bB^\top (\bA^\top)^{r-1} \bc \bc^\top ( \bA^{r-1} \! + \! \alpha \bA_0^{-1} \! + \! \sum_{j=1}^{r-1} \beta_j \bA^{j-1} ) }{ \alpha \xi }
    \end{align*}
    Now, using again~\eqref{eq:simplification-expression} 
    in $\eta(-\tilde{\bA}^{-1}\tilde{\bb})$, we get 
    \begin{align*}
        \eta(-\tilde{\bA}^{-1}\tilde{\bb}) \! = \! -\bc^\top \bA^{r-1} (\bI - \bDelta) \tilde{\bb} \! - \! \alpha \bc^\top \bA_0^{-1} (\bI - \bDelta)\tilde{\bb} \! + \! \alpha d.
    \end{align*}
    After some computations, we can show that
    \begin{align*}
        \bc^\top \bA^{r-1}(\bI-\bDelta)\tilde{\bb} = 
        \frac{d}{\xi} \theta^2, \ 
        \bc^\top \bA_0^{-1} (\bI-\bDelta)\tilde{\bb} = d,
    \end{align*}
    Hence, $\eta(-\tilde{\bA}^{-1}\tilde{\bb}) = -\frac{d \theta^2}{\xi}$, and the result follows.
\end{proof}

As shown in Lemma 4.1 of~\cite{PM-YC-EDA-JC:25-arxiv}, $\tilde{\bp}\in\partial\Cc_{r-1}$. 
Interestingly, 
the existence of undesired equilibria for~\eqref{eq:closed-loop-piecewise-affine} can also be characterized in terms of the parity of the positive real eigenvalues of $\tilde{\bA}$.

\begin{proposition}\longthmtitle{Undesired equilibria and positive real eigenvalues}\label{prop:existence-undes-eq-alternative-condition}
    Suppose that $\tilde{\bA}$ is invertible.
    Then, the origin is the only equilibrium of~\eqref{eq:closed-loop-piecewise-affine} if and only if the number of positive real eigenvalues of $\tilde{\bA}$ is even.
\end{proposition}
\begin{proof}
    As shown in Lemma~\ref{lem:conditions-tildeA-invertible}, the eigenvalues of $\tilde{\bA}$ that are not $\{ -\alpha_i \}_{i=1}^r$ and are not eigenvalues of $\bA_0$ satisfy $H(\lambda) := \bc^\top (\lambda\bI - \bA_0)^{-1} \bB \bG^{-1} \bB^\top (\bA^\top)^{r-1} \bc = 0$.
    Note that for $\lambda > \rho(\bA_0)$, it holds that (cf.~\cite[5.6.P26]{RAH-CRJ:12})
    $(\lambda \bI - \bA_0)^{-1} = \frac{1}{\lambda}\sum_{i=0}^{\infty} \frac{\bA_0^i}{\lambda^{i}}$.
    Using the fact that 
    $\bc^\top \bA^i \bB = \mathbf{0}_m$ for $i\in[r-2]$ 
    it follows that $\bc^\top \bA_0^i \bB = \mathbf{0}_m$ for $i\in[r-2]$.
    Hence, for $\lambda > \rho(\bA_0)$,
    it holds that
    $H(\lambda) = \frac{ \theta^2 }{\lambda^r} + \mathcal{O}(\frac{1}{\lambda^{r+1}})$.
    On the other hand, $(\lambda\bI - \bA_0)^{-1} = \frac{ \text{adj}(\lambda\bI - \bA_0) }{ \text{det}(\lambda\bI - \bA_0) }$. Hence, 
    \begin{align}\label{eq:f-lambda-rational}
        H(\lambda) = \frac{\sigma_1 \prod_{i=1}^{n-1} (\lambda-z_i) }{ \sigma_2 \prod_{i=1}^{n} (\lambda-p_i) },
    \end{align}   
    where $\{ z_i \}_{i=1}^{n-1}$ (resp. $\{ p_i \}_{i=1}^n$) are the zeros (resp. poles) of $H$.
    Since $H(\lambda) = \frac{ \theta^2 }{\lambda^r} + \mathcal{O}(\frac{1}{\lambda^{r+1}})$, there must be $n-r-1$ zero-pole cancellations in~\eqref{eq:f-lambda-rational}, and $H(\lambda) = \frac{\sigma_1}{\sigma_2 \lambda^r} + \mathcal{O}(\frac{1}{\lambda^{r+1}})$ for $\lambda > \rho(\bA_0)$, which implies that $\frac{\sigma_1}{\sigma_2} = \theta^2 > 0$.
    Note also that the poles of $H$ are the eigenvalues of $\bA_0$, and since $\bA_0$ is Hurwitz, $\prod_{i=1}^n (-p_i) > 0$ (indeed, conjugate pairs contribute as a positive factor to the product and real eigenvalues also contribute with a positive factor because they are negative).
    Similarly, (since none of the zeros can be zero because $\tilde{\bA}$ is invertible) the sign of $\prod_{i=1}^n (-z_i)$ is positive if and only if the number of positive real zeros is even (indeed, the negative real and conjugate pairs contribute to a positive factor in the product).
    Hence, from~\eqref{eq:f-lambda-rational}, the sign of $H(0) = -\xi$ is equal to the number of positive real zeros of $H$.
    By {\color{black}Proposition}~\ref{prop:existence-of-undesired-equilibrium}, the origin is the only equilibrium of~\eqref{eq:closed-loop-piecewise-affine} if and only if the number of positive real eigenvalues of $\tilde{\bA}$ is even.
\end{proof}

\section{Stability of the origin}\label{sec:stability-origin}

%

\subsection{Global Exponential Stability}

Here we study conditions under which the origin is globally exponentially stable (GES) under~\eqref{eq:closed-loop-piecewise-affine}.
The following is the main result of this section:

\begin{theorem}\longthmtitle{GES}\label{thm:global-exponential-stability}
    If $\tilde{\bA}$ is Hurwitz {\color{black}
    and $\bA_0 \tilde{\bA}$ has no negative real eigenvalues}, then the origin is GES under~\eqref{eq:closed-loop-piecewise-affine}.
\end{theorem}
\begin{proof}
    {\color{black}
    Since $\bA_0, \tilde{\bA}$ are Hurwitz, $\bA_0 -\tilde{\bA}$ is rank 1 and $\bA_0 \tilde{\bA}$ has no negative real eigenvalues~\cite[Theorem 1]{CK-MN:06} shows that $\bA_0, \tilde{\bA}$ have a CQLF of the form $V(\bx) = \bx^\top \bP \bx$.
    (In fact, the assumptions are necessary and sufficient for such a CQLF to exist).}
    Next, we use~\cite[{\color{black}Theorem} 1]{AP-NV-HN:05}.
    Note that~\eqref{eq:closed-loop-piecewise-affine} is a piecewise affine system of the form considered in~\cite[Section III]{AP-NV-HN:05},
    and by an argument analogous to the one in~\cite[{\color{black}Theorem} 2]{XX-PT-JWG-ADA:15},~\eqref{eq:closed-loop-piecewise-affine} is locally Lipschitz (and hence continuous).
    Now, $\bP$ satisfies $\bP \bA_0 + \bA_0^\top \bP \prec 0$ and $\bP \tilde{\bA} + \tilde{\bA}^\top \bP \prec 0$.
    Hence, by~\cite[{\color{black}Theorem} 1]{AP-NV-HN:05},~\eqref{eq:closed-loop-piecewise-affine} is exponentially convergent (cf.~\cite[Definition 1]{AP-NV-HN:05}). Since $\bx(t) \equiv 0$ for all $t\geq 0$ is a solution of~\eqref{eq:closed-loop-piecewise-affine}, 
    by~\cite[Definition 1]{AP-NV-HN:05} the origin is GES.
\end{proof}

{\color{black}
We note that even if $\tilde{\bA}$ is Hurwitz, $\bA_0 \tilde{\bA}$ might have negative real eigenvalues. This is the case for the example in~\cite[Section III.B]{NM-JC-PS-KZ:25}, for which the origin is not GES.}
%
{\color{black}
Given Theorem~\ref{thm:global-exponential-stability}, a natural question is how to design $\bK$ so that the conditions therein are satisfied.
We next show how such $\bK$ can be obtained by solving a pair of LMIs.}

\begin{lemma}\longthmtitle{Nominal controller design as a pair of LMIs}\label{lem:nominal-controller-pair-of-lmis}
    Let $\hat{\bA} = \bA + \bv_1 \bc^\top \phi(\bA)$, $\hat{\bB} = ( \bI + \bv_1 \bc^\top \bA^{r-1} ) \bB$.
    Then, the LMIs in $\bQ$, $\bY$ defined by
    \begin{subequations}
    \begin{align}
        \bA \bQ + \bQ \bA^\top + \bB \bY + \bY^\top \bB^\top \prec 0, \\
        \hat{\bA} \bQ + \bQ \hat{\bA}^\top + \hat{\bB} \bY + \bY^\top \hat{\bB}^\top \prec 0,
    \end{align}
    \label{eq:lmis-common-LF}
    \end{subequations}
    are feasible if and only if there exists a gain $\bK$ such that the corresponding matrices $\bA_0$ and $\tilde{\bA}$ admit a CQLF. 
    In this case, $\bP = \bQ^{-1}$, $\bK = -\bY \bQ^{-1}$ are such that $\bx^\top \bP \bx$ is a CQLF for $\bA_0, \tilde{\bA}$, and the origin is GES for~\eqref{eq:closed-loop-piecewise-affine}.
\end{lemma}
\begin{proof}
    Using the change of variables in~\cite[Section 7.2.1]{SB-LEG-EF-VB:94}, $\bQ$, $\bY$ satisfy~\eqref{eq:lmis-common-LF} if and only if $\bK=-\bY\bQ^{-1}$, $\bP=\bQ^{-1}$ simultaneously satisfy the Lyapunov equations for the pairs $(\bA,\bB)$ and $(\hat{\bA},\hat{\bB})$, {\color{black}
    i.e., 
    \begin{align*}
        (\bA - \bB \bK)^\top \bP + \bP (\bA - \bB \bK) \prec 0, \\
        (\hat{\bA} - \hat{\bB} \bK)^\top \bP + \bP (\hat{\bA} - \hat{\bB} \bK) \prec 0.
    \end{align*}}
    {\color{black}
    Since 
    $\hat{\bA}-\hat{\bB}\bK = \bA - \bB \bK + \bv_1^\top \bc^\top \phi(\bA) - \bv_1 \bc^\top \bA^{r-1} \bB \bK = \tilde{\bA}$,
    }
    $\bK = -\bY\bQ^{-1}$ ensures that $\bx^\top \bP \bx$ {\color{black}is a CQLF for $\tilde{\bA}, \bA_0$.
    As shown in the proof of Theorem~\ref{thm:global-exponential-stability}, this CQLF exists (and hence~\eqref{eq:lmis-common-LF} is feasible) if and only if the assumptions of Theorem~\ref{thm:global-exponential-stability} hold.
    Finally, by Theorem~\ref{thm:global-exponential-stability}, such $\bK$ renders the origin GES for~\eqref{eq:closed-loop-piecewise-affine}.}
\end{proof}

\subsection{Unbounded Trajectories}\label{sec:unbounded-trajectories}

Here we study the existence of unbounded trajectories for~\eqref{eq:closed-loop-piecewise-affine}.
Although~\cite{PM-YC-EDA-JC:25-arxiv,JJC-CJT-SS-KS:25} already show examples where such trajectories exist, here we provide a simple condition that guarantees their existence for system~\eqref{eq:closed-loop-piecewise-affine}.



\begin{proposition}\longthmtitle{Unbounded trajectories}\label{prop:unbounded-trajectories}
    Suppose that Assumptions~\ref{as:linear-dynamics-affine-constraints}-\ref{as:stabilizability} hold.
    If $\tilde{\bA}$ has positive real eigenvalues, then~\eqref{eq:closed-loop-piecewise-affine} has unbounded trajectories.
\end{proposition}
\begin{proof}
    Let $\bv\in\real^n$ be an eigenvector of $\tilde{\bA}$ with positive real eigenvalue $\lambda>0$.
    %
    %
    %
    Since $\bc^\top \phi_{r-1}(\bA)$ is a left eigenvector of $\tilde{\bA}$ (cf. Lemma~\ref{lem:conditions-tildeA-invertible}) 
    and $\lambda\neq-\alpha_r$ (because $\alpha_r > 0$), by~\cite[{\color{black}Theorem} 1.4.7]{RAH-CRJ:12}, $\bc^\top \phi_{r-1}(\bA)\bv = 0$.
    Next, consider the
    curves $\bx_{+}(t) = \tilde{\bp} + e^{\lambda t} \bv$, $\bx_{-}(t) = \tilde{\bp} - e^{\lambda t} \bv$.
    Since $\tilde{\bp}\in\partial\Cc_{r-1}$ 
    (cf.~\cite[Lemma 4.1]{PM-YC-EDA-JC:25-arxiv})
    and $\bc^\top \phi_{r-1}(\bA) \bv = 0$, both curves are contained in $\partial\Cc_{r-1}$ for all $t\in\real$.
    Let us further show that such curves are not contained in a hyperplane parallel to $\eta(\bx) = 0$. 
    Indeed, if that was the case, $\bv$ would have to be perpendicular to $\bv_2$ (the normal vector defining the hyperplane $\eta(\bx) = 0$).
    Hence, $\bv_2^\top \bv = 0$.
    By using this property, we get $\lambda \bv = \tilde{\bA}\bv = \bA_0\bv$,
    which implies that $\bA_0$ has an eigenvector with positive real eigenvalue, contradicting the assumption that $\bA_0$ is Hurwitz.
    Therefore, since $\{ \bx_{+}(t) \}_{t\geq0}$ and $\{ \bx_{-}(t) \}_{t\geq0}$ are non-intersecting half-lines and $\eta(\bx) = 0$ is a hyperplane,
    this implies that there exists $T > 0$ sufficiently large such that for $t \geq T$, either $\bx_{+}(t)\in\Rc_{-}$ or $\bx_{-}(t)\in\Rc_{-}$ holds.
    The corresponding curve is a trajectory of~\eqref{eq:closed-loop-piecewise-affine} and is unbounded.
\end{proof}


{\color{black}Proposition}~\ref{prop:unbounded-trajectories} provides a simple test on $\tilde{\bA}$ that ensures that~\eqref{eq:closed-loop-piecewise-affine} has unbounded trajectories.
In fact, its proof shows that such trajectory diverges while staying on the hyperplane $\partial\Cc_{r-1}$.
%
Note also that the case where $\tilde{\bA}$ is unstable but the unstable modes correspond to complex conjugate eigenvalues with positive real part is not covered by {\color{black}Proposition}~\ref{prop:unbounded-trajectories}.
As we show in Section~\ref{sec:simulations}, this case can lead to unbounded trajectories or the origin being GES.
%
%

Interestingly, dynamical properties can be linked to invariant zeros of an appropriate
SISO system.

\begin{corollary}\longthmtitle{Invariant zeros}\label{prop:global-exponential-stab-minimum-phase-siso}
    If all invariant zeros of the SISO system $\mathcal{S} = (\bA_0, \bB \bG^{-1} \bB^\top (\bA^\top)^{r-1} \bc, \bc^\top)$ have negative real part {\color{black}and $\bA_0 \tilde{\bA}$ has no negative real eigenvalues}, the origin is GES for~\eqref{eq:closed-loop-piecewise-affine}.
    If $\mathcal{S}$ has positive real invariant zeros, then~\eqref{eq:closed-loop-piecewise-affine} has unbounded trajectories.
\end{corollary}
\begin{proof}
    Follows from {\color{black}Theorem}~\ref{thm:global-exponential-stability} and {\color{black}Proposition}~\ref{prop:unbounded-trajectories}
    by observing that invariant zeros of $\mathcal{S}$ are exactly solutions of~\eqref{eq:eigenvalues-of-tildeA}, which as shown in Lemma~\ref{lem:conditions-tildeA-invertible}, are exactly the eigenvalues of $\tilde{\bA}$ that are not $\{ -\alpha_i \}_{i=1}^r$ or eigenvalues of $\bA_0$.
\end{proof}

\begin{remark}
\label{cor:dynamical-properties-single-input-systems}
    If $m=1$, $\bB^\top (\bA^\top)^{r-1} \bc, \bG^{-1} \in \real$.
    Since $\tilde{\bA} = \bA - \frac{\bB \bc^\top \phi(\bA) }{ \bB^\top (\bA^\top)^{r-1} \bc }$
    and $\text{spec}(\tilde{\bA})\backslash \{ -\alpha_i \}_{i=1}^r$ is
    independent of $\{ \alpha_i \}_{i=1}^r$ (cf.
    Lemma~\ref{lem:conditions-tildeA-invertible}),
    the guarantees in {\color{black}Propositions~\ref{prop:existence-undes-eq-alternative-condition}} and~\ref{prop:unbounded-trajectories} are independent of $\bK, \{ \alpha_i \}_{i=1}^r$, and $\bG$.
    %
    %
    \demo
\end{remark}


\begin{figure}
    \centering
    \includegraphics[width=0.8\linewidth]{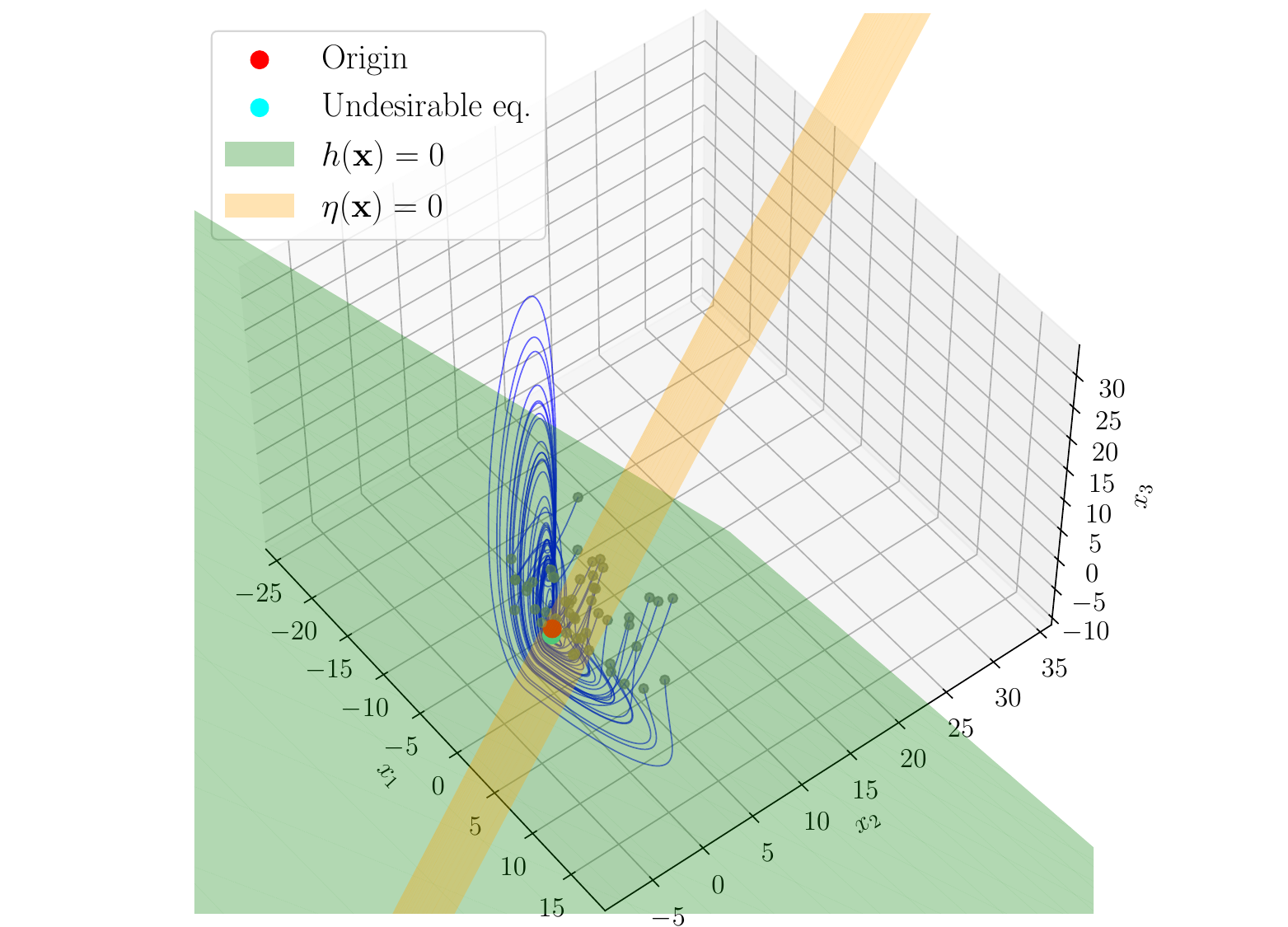}
    \includegraphics[width=0.8\linewidth]{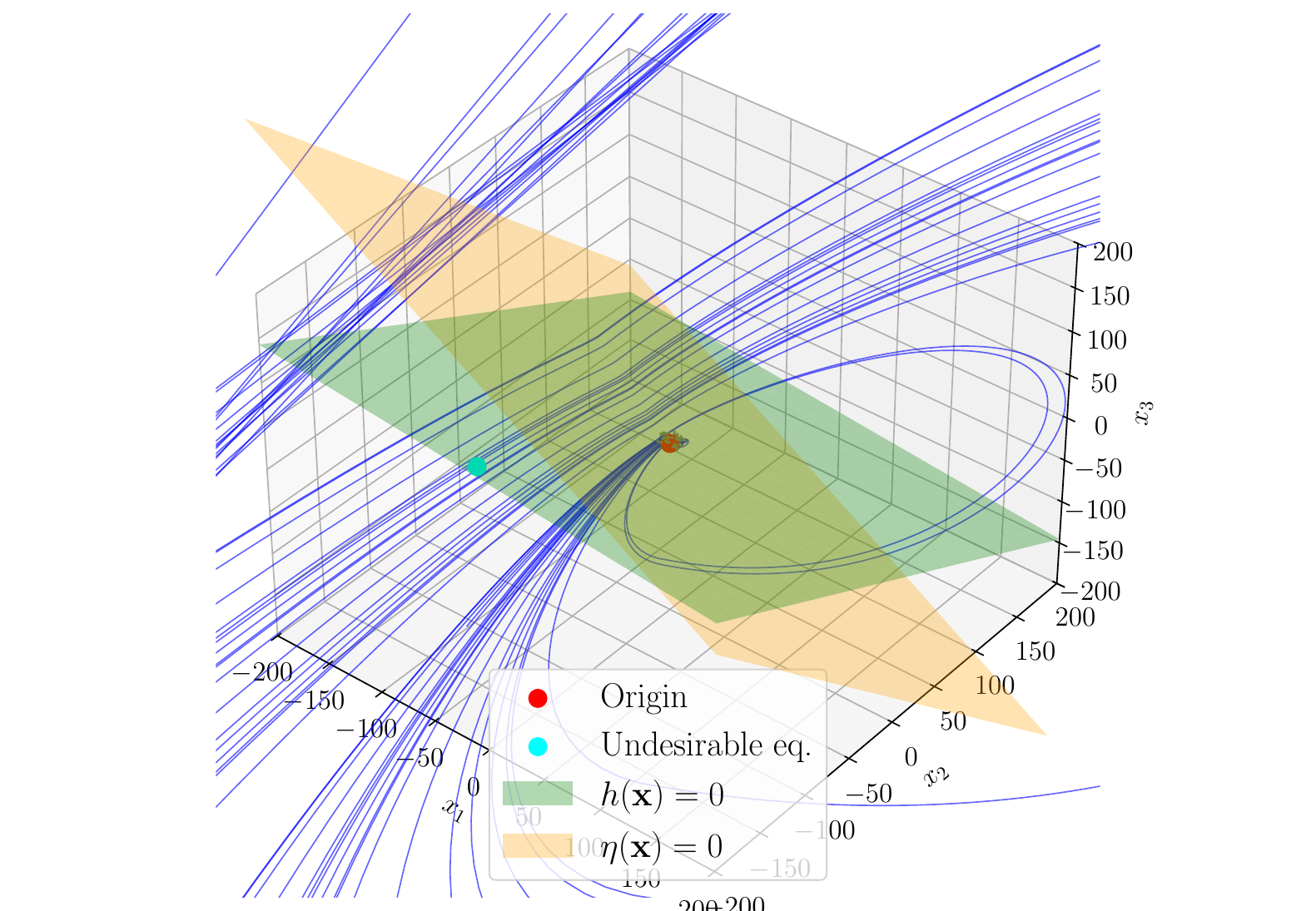}

    \caption{
    \footnotesize
    Trajectories for two systems with $n=3$, $m=1$, and $\tilde{\bA}$ having a pair of complex conjugate eigenvalues with positive real part. (Top) The origin is GES. $\bA = [ 0.65, 1.18, 0.05; 0.38, 0.93, -0.7; 1.52, 1.12, 0.22]$, $\bB = [-1.24; 1.93; -0.63]$, $\bK = [4.57, 6.23, -0.01]$, $\bc = [-0.1; 1.32; 0.67]$, $d = 0.71$. (Bottom) Unbounded trajectories. $\bA = [ 0.45, -1.47, 1.48; 0.47, -0.12, -0.57; 0.99, -0.11, -2.5]$, $\bB = [0.06; -0.31, 0.19]$, $\bK = [3.09, -4.2, 2.12]$, $\bc = [0.31; 1.32; 2.26]$, $d = 2.5$. Both examples use $\alpha = 5$, $\bG = 1$ and have $r=1$.
    }
    \label{fig:two-behaviors}
\end{figure}

\section{Simulations}\label{sec:simulations}

In this section we showcase the results of the paper in different examples.

First, we note that the two bottom figures in Figure~\ref{fig:hero} showcase the two possible dynamical behaviors studied in the paper for two different planar systems.
Both examples use $\alpha = 5$, $\bG = 1$ and have $r=1$.
The bottom right corresponds to an example where the origin is GES, with $\bA = -[ 0.79, 1.6; 0.43, 0.01]$, $\bB = [0.61; 0.55]$, $\bK = [0.33, 0.88]$, $\bc = -[0.26, 0.86]$, $d = 0.49$, and {\color{black}the conditions in Theorem~\ref{thm:global-exponential-stability} are satisfied}.
The bottom left corresponds to an example with unbounded trajectories, with $\bA = [ 1.43, -0.97; 1.26, -0.04]$, $\bB = [0.23; -1.04]$, $\bK = [5.3, 0.9]$, $\bc = [-0.64, 0.9]$, $d = 0.36$, and the associated $\tilde{\bA}$ having a positive real eigenvalue.

{\color{black}
There are various examples not covered by the results in the paper. For example, as mentioned in Section~\ref{sec:stability-origin} the origin is not GES for the example in~\cite[Section III.B]{NM-JC-PS-KZ:25} for which $\tilde{\bA}$ is Hurwitz but $\tilde{\bA} \bA_0$ has negative real eigenvalues (and hence the assumptions of Theorem~\ref{thm:global-exponential-stability} are violated).}
Similarly, Figure~\ref{fig:two-behaviors} illustrates two systems for which $\tilde{\bA}$ has a pair of complex conjugate eigenvalues with positive real part,
which can lead to the origin being GES (top) or the trajectories being unbounded (bottom).

Next, we test the results in the paper for the safe control of the roll-yaw dynamics of a mid-size aircraft around an operating point defined by velocity $717.17$ ft/sec, altitude 25000 ft, and angle of attack $4.5627^\circ$~\cite[Section 14.8]{EL-KAW:24}.
{\color{black}
Although in this example the safety filter is applied to an affine (instead of linear) system, it can be shown that the results in the paper can easily be extended to this setting.}
The state of the system is $\bx_p = [\beta, p_s, r_s]^\top$, where $\beta$ is the sideslip angle (in rad), and $p_s, r_s$ are roll and yaw rates (in rad/s). The inputs are aileron and rudder deflections $\delta_a$ and $\delta_r$ (rad). The plant dynamics are $\dot{\bx}_p = \bA_p \bx_p + \bB_p\bu$,
with $\bA_p$, $\bB_p$ as in~\cite[Section 5.2]{PM-SSM-PO-LY-ED-JWB-ADA:25}.
Our goal is to regulate $p_s$ to a desired commanded signal $y_{\text{cmd}}:\real_{\geq0}\to\real$ (cf. Figure~\ref{fig:sim-plot} (top)). Note that $p_s = \bC_p \bx_p$, with $\bC_p = [0, 1, 0]$ and $r=1$. 
To do so, we use a 
method similar to that in~\cite[Section 4.4.1]{EL-KAW:24} and design a
nominal LQR PI controller using the the integrated output tracking error dynamics, with 
{\color{black}$\bQ = \text{diag}(11.67, 35.75, 10.89, 0.09)$
and $\bR = \text{diag}(1, 0.02)$}.
The dynamics of the error variable are $\dot{e}_{yI} = \bC_p \bx_p - y_{\text{cmd}} - \kappa e_{yI}$,
where $\kappa = 0.01$ is a regularization factor.
The system in the $[e_{yI}, \bx_p]$ variables is:
{\small
\begin{align}\label{eq:extended-dynamics}
    \begin{bmatrix}
        \dot{e}_{yI} \\
        \dot{\bx}_p
    \end{bmatrix} = 
    \underbrace{
    \begin{bmatrix}
        -\kappa & \bC_{p} \\
        \mathbf{0}_{3} & \bA_p
    \end{bmatrix} 
    }_{\bA}
    \begin{bmatrix}
        e_{yI} \\
        \bx_p
    \end{bmatrix}
    + 
    \underbrace{
    \begin{bmatrix}
        \mathbf{0}_2^\top
        \\
        \bB_p
    \end{bmatrix} 
    }_{\bB}
    \bu
    + 
    \begin{bmatrix}
        -p_{\text{cmd}} \\
        \mathbf{0}_3
    \end{bmatrix}.
\end{align}
}

Next, we set an upper bound on the roll rate limit equal to $0.4$, leading to a CBF with $\bc = [0, 0, -1, 0]$ and $d = 0.4$.
Although~\eqref{eq:extended-dynamics} is affine (instead of linear) in the state variables 
the results derived throughout the paper carry over analogously, with the difference that closed-loop equilibrium is not the origin and is affected by the affine term in~\eqref{eq:extended-dynamics}.
In order to guarantee that such equilibrium is GES, we design $\bK$ to satisfy the conditions in Theorem~\ref{thm:global-exponential-stability} (with $\bA$ and $\bB$ as in~\eqref{eq:extended-dynamics}, $\alpha=1$, and $\bG=\bI_2$).
In this case, we use the associated LQR gain, which satisfies this. 
Figure~\ref{fig:sim-plot} showcases the evolution of the roll rate $p_s$ under the nominal controller and the safety filter.

\begin{figure}
    \centering
    \includegraphics[width=0.83\linewidth]{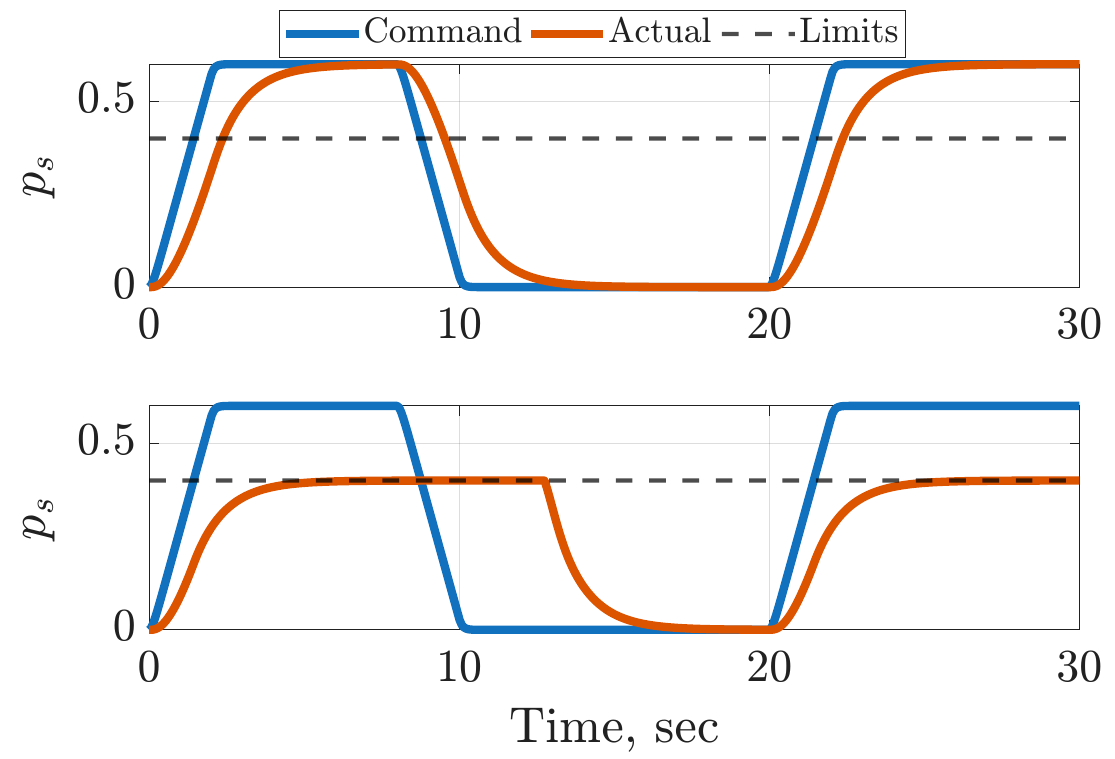}
    
    \caption{
    \footnotesize
    Evolution of $p_s$ for the filtered (bottom) and nominal (top) systems.
    }
    \label{fig:sim-plot}
\end{figure}

{\color{black}
\section{Acknowledgements}

The authors thank Prof. Joaquin Carrasco for bringing to their attention the example in~\cite[Section III.B]{NM-JC-PS-KZ:25}.}

\section{Conclusions}

We have studied the dynamical properties of CBF-based \textit{safety filters} for linear systems and affine constraints.
We have derived conditions under which the closed-loop system has undesired equilibria, unbounded trajectories, or the origin is GES.
Our results provide simple design conditions on the safety filter parameters to ensure that it induces desirable dynamical properties. Future work will
study the dynamical behavior in the cases where the assumptions of Theorem~\ref{thm:global-exponential-stability} and Proposition~\ref{prop:unbounded-trajectories} are violated,
and extend the results to multiple affine constraints and nonlinear systems.

\bibliography{bib/alias,bib/Main-add,bib/Main,bib/JC,bib/New,bib/PM}
\bibliographystyle{IEEEtran}

\end{document}